\documentclass[12pt]{amsart}

\usepackage{amsmath}
\usepackage{amsthm}
\usepackage{amssymb}
\usepackage{lscape,xcolor}
\usepackage{graphicx}
\usepackage{mathrsfs}
\usepackage{stmaryrd}
\usepackage{verbatim}
\usepackage{rotating}
\usepackage{tikz-cd}
\usepackage{amsrefs}
\usepackage{hyperref}
\usepackage{euscript}
\usepackage[colorinlistoftodos]{todonotes}

\allowdisplaybreaks[1]

\newcommand{\mmod}{\! \sslash \!}

\newcommand{\ull}[1]{\underline{#1}}

\newcommand{\mr}[1]{\mathrm{#1}}

\newcommand{\bra}[1]{\langle #1 \rangle}
\newcommand{\br}[1]{\overline{#1}}

\newcommand{\Z}{\mathbb{Z}}

\newcommand{\F}{\mathbb{F}}

\DeclareMathOperator{\coker}{\mathrm{coker}}

\newcommand{\tBP}[1]{BP\bra{#1}}

\newcommand{\HZ}{\mr{H}\Z}
\def \HF2{\mr{H}\F_2}

\DeclareMathOperator{\Ext}{Ext}

\newcommand\floor[1]{\lfloor#1\rfloor}

\newcommand{\HZu}{\ull{\HZ}}

\newcommand{\tBPu}[1]{\ull{\tBP{#1}}}

\def \AA0{\br{A \mmod A(0)}_*}
\def \AA2{A\mmod A(2)_*}
\def \AE2{(A\mmod E(2))_*}
\renewcommand{\AE}[1]{A\mmod E(#1)_*}
\DeclareMathOperator{\wt}{\mathrm{wt}}
\def \E2E1{E(2)\mmod E(1)_*}
\newcommand{\otau}{\overline{\tau}}
\newcommand{\ellu}{\ull{\ell}}

 \newtheorem{thm}[equation]{Theorem}
 \newtheorem{cor}[equation]{Corollary}
 \newtheorem{lem}[equation]{Lemma}
 \newtheorem{prop}[equation]{Proposition}
 \newtheorem{obs}[equation]{Observation}

 \newtheorem*{thm*}{Theorem}
 \newtheorem*{cor*}{Corollary}
 \newtheorem*{lem*}{Lemma}
 \newtheorem*{prop*}{Proposition}
  \newtheorem*{not*}{Notation}

 \theoremstyle{definition}
 \newtheorem{defn}[equation]{Definition}
 \newtheorem{ex}[equation]{Example}
 \newtheorem{exs}[equation]{Examples}
 \newtheorem{rmk}[equation]{Remark}

\newtheorem*{defn*}{Definition}
\newtheorem*{ex*}{Example}
\newtheorem*{exs*}{Examples}
\newtheorem*{rmk*}{Remark}
\newtheorem*{claim*}{Claim}

\numberwithin{equation}{section}
\numberwithin{figure}{section}

\author{D.~ Culver}\address{University of Illinois, Urbana-Champaign}\email{dculver@illinois.edu}

\thanks{This work was partially supported by NSF grant DMS-1547292}

\title{The $\tBP{2}$-cooperations algebra at odd primes}

\begin{document}

\maketitle

\begin{abstract}
	In previous work, the author analyzed the co-operations algebra for the second truncated Brown-Peterson spectrum at the prime $p=2$. The purpose of this paper is to carry out the necessary modifications to odd primes.
\end{abstract}

\tableofcontents

\section{Introduction}

In recent work, \cite{Culver2017}, the author studied the cooperations algebra for the second truncated Brown-Peterson spectrum at the prime 2. The purpose of this paper is carry out the analogous analysis at odd primes. Many of the proofs carry through in a similar way as in the $p=2$ case, though there are some differences that arise due to the difference in the structure of the dual Steenrod algebra at odd primes. 

The structure of this paper is similar to what is found in \cite{Culver2017}. In section \ref{sec:prereqs} we give an overview of the dual Steenrod algebra at odd primes. We also discuss the homology of the truncated Brown-Peterson spectra $\tBP{n}$ at odd primes and show how to derive the their coaction over the quotient $E(\otau_0, \ldots \otau_n)$ of the dual Steenrod algebra. 

In section \ref{sec:structural results}, we prove the analogous splitting results of subalgebra $\AE{2}$ as an $E(\otau_0, \otau_1,\otau_2)$-comodule. We also show that the $E_2$-term of the following Adams spectral sequence
\[
\Ext_{E(2)_*}(\F_2, \AE{2})\implies \pi_*(\tBP{2}\wedge \tBP{2})^\wedge_p
\]
decomposes into two summands: one which is concentrated in Adams filtration 0 and another which is $v_2$-torsion free and concentrated in even degrees. This results in the collapsing of the Adams spectral sequence for $\tBP{2}\wedge \tBP{2}$, just as it does in the $p=2$ case. 

In section \ref{sec:weight}, we produce the analogues of weight and Brown-Gitler comodules to our odd primary case. We also derive analogues of the exact sequences relating these Brown-Gitler comodules. 

Finally, in section \ref{sec:calculations} we provide an inductive scheme for computing the Ext groups of these Brown-Gitler comodules. We provide charts for the case when $p=3$.

\subsubsection*{Conventions} Throughout this note, $p$ will always denote an odd prime. We will use $H$ to denote the Eilenberg-MacLane spectrum $H\F_p$. We also set $q:=2(p-1)$. Given a Hopf algebra $B$ and a comodule $M$ over $B$, we will often abbreviate $\Ext_B(\F_2,M)$ to $\Ext_B(M)$. All spectra are implicitly $p$-complete. 

Also, we will use the notation $E(n)$ to denote the subalgebra of $A$ generated by the Milnor primitives $Q_0, \ldots, Q_n$. This is in conflict with the standard notation for the Johnson-Wilson theories, but as these never arise in this paper, this will not present an issue.

\section{Odd primary prerequisites}\label{sec:prereqs}

The purpose of this section is to provide, for the readers' convenience, various facts about the odd primary dual Steenrod algebra. We also record the mod $p$ homology of the spectra $\tBP{n}$ and their Margolis homology. 

\subsection{Review of the dual Steenrod algebra}

Let $p$ denote an odd prime. Then the dual $p$-primary Steenrod algebra is 
\[
A_* = \F_p[\xi_1, \xi_2, \xi_3, \ldots]\otimes E(\tau_0, \tau_1, \tau_2, \ldots)
\]
where 
\[
|\xi_n| = 2(p^n-1)
\]
and 
\[
|\tau_i|=2p^i-1.
\]
We let $P_*$ denote the polynomial part and $E_*$ the exterior part. So $A_*\cong P_*\otimes E_*$. The coproduct is given by the following formulas (cf. \cite{Milnordual}):
\begin{equation}\label{eqn2:coprod-xis}
\psi(\xi_n)=\sum_{i+j=n}\xi_i^{p^j}\otimes \xi_j
\end{equation}
and 
\begin{equation}\label{eqn2:coprod-taus}
	\psi(\tau_n) =\tau_n\otimes 1+ \sum_{i+j=n}\xi_{i}^{p^j}\otimes\tau_j.
\end{equation}
Note that $P_*$ forms a sub Hopf algebra and that there is a short exact sequence
\[
0\to A_*\cdot \overline{P_*}\to A_*\to E_*\to 0
\]
where $\overline{P_*}$ denotes the augmentation ideal of $P_*$. This endows $E_*$ the structure of a primitively generated exterior Hopf algebra.

The dual Steenrod algebra also has a canonical anti-automorphism, the conjugation map 
\[
\chi: A_*\to A_*.
\]
In \cite{Milnordual}, it is shown that $\chi$ satisfies the following formula
\begin{equation}\label{eqn: conj on zeta}
	\sum_{i+j=n}\xi_j^{p^i}\chi\xi_i=0
\end{equation}
and 
\begin{equation}\label{eqn: conj on tau}
	\tau_n+\sum_{i+j=n}\xi_j^{p^i}\chi \tau_i = 0
\end{equation}
In the dual Steenrod algebra, we will let $\zeta_n:=\chi\xi_n$ and $\otau_n:= \chi\tau_n$. The coproducts on these elements are then given by 
\begin{equation}\label{eqn2:coprod-zetas}
	\psi(\zeta_n) = \sum_{i+j=n}\zeta_j\otimes \zeta_i^{p^j}
\end{equation}
and 
\begin{equation}\label{eqn2:coprod-otaus}
\psi(\otau_n) = 1\otimes \otau_n + \sum_{i+j=n}	\otau_j\otimes \zeta_i^{p^j}.
\end{equation}

In \cite{Milnordual}, Milnor identified an important family of elements in the Steenrod algebra $A$ which will play an important role in this work. 

\begin{defn}
	The \emph{Milnor primitives} are defined inductively as elements $Q_k\in A$ by 
	\begin{align*}
		Q_0&:= \beta &  Q_{k+1}&:=[P^{p^k}, Q_k]
	\end{align*}
	where $\beta$ denotes the Bockstein element and 
	\[
	[a,b]:= ab-(-1)^{|a|\cdot |b|}ba.
	\]
\end{defn}

Milnor showed that the $Q_k$ are dual to $\otau_k$ and that they generate a commutative primitively generated exterior Hopf algebra $E$ inside the Steenrod algebra. Because primitively generated exterior algebras are self dual, the dual of $E$ is also a primitively generated exterior Hopf algebra. In fact, its dual is $E_*$. 

\begin{defn}
	For $n\geq 0$, let $E(n)$ denote the subalgebra of the Steenrod algebra generated by the Milnor primitives $Q_0, \ldots Q_n$,
	\[
	E(n):= E(Q_0, \ldots , Q_n).
	\]
\end{defn}

Then the dual of $E(n)$ is the primitively generated exterior Hopf algebra
\[
E(n)_* = E(\otau_0, \ldots , \otau_n).
\]

\begin{rmk}
	Note that it follows from \eqref{eqn: conj on tau} that the elements $\tau_n$ and $\otau_n$ are congruent modulo the ideal $A_*\cdot \overline{P}_*$. So 
	\[
	E(\otau_0, \ldots , \otau_n) = E(\tau_0, \ldots, \tau_n).
	\]
\end{rmk}

\subsection{mod $p$ homology of $\tBP{n}$}

In this subsection, we record the homology of the truncatd Brown-Peterson spectra $\tBP{n}$ and describe the coaction on these comodules. The homology of the truncated Brown-Peterson spectra was computed by Wilson in \cite{OmegaWilson}. 
\begin{thm}[Wilson, \cite{OmegaWilson}]
	\[
	H_*\tBP{n}= \AE{n}
	\]
\end{thm}

Since the Milnor primitive $Q_k$ is dual to $\tau_k$ (cf. \cite{Milnordual}), we have that 
\[
\AE{n} = P_*\otimes E(\otau_{n+1}, \otau_{n+2}, \ldots ).
\]
In order to proceed, we need to know $\AE{n}$ as a comodule over the dual Steenrod algebra. As in \cite{Culver2017}, this is obtained as follows:

\[
\begin{tikzcd}
\alpha:	\AE{n}\arrow[r,"\psi"] & A_*\otimes \AE{n}\arrow[r, "\pi\otimes 1"] & E(n)_*\otimes \AE{n}
\end{tikzcd}
\]
where the first map is the restriction of the coproduct to $\AE{n}$. That it's target is $A_*\otimes \AE{n}$ follows from applying the coproduct formulas to $\zeta_n$ and $\otau_k$ for $k\geq 3$. The map $\pi$ is the projection map fitting in the following short exact sequence
\[
\begin{tikzcd}
	0\arrow[r] & A_*\cdot \overline{\AE{n}}\arrow[r] & A_* \arrow[r,"\pi"] & E(n)_*\arrow[r] & 0.
\end{tikzcd}
\]

Since $\tBP{n}$ is a commutative ring spectrum, its homology is a comodule algebra. So the coaction map will be an algebra map, and so it is enough to determine the coaction on the algebra generators. 	From \cite{Milnordual}, the Hopf algebra $E(n)_*$ is the exterior algebra $E(\otau_0, \ldots, \otau_n)$ where the generators $\otau_k$ are primitive. Thus we have the following,

\begin{prop}\label{prop:E(n)coaction}
The $E(n)_*$-coaction on $\AE{n}$ is given on generators by
\begin{align}
	\alpha(\zeta_k) &= 1\otimes \zeta_k \label{eqn:Enzeta}\\
	\alpha(\otau_{n+k}) &= 1\otimes \otau_{n+k}+\sum_{0\leq i\leq n}\otau_i\otimes \zeta_{n+k-i}^{p^i} \label{eqn:Enotau}
\end{align}
for all $k>0$.
\end{prop}

This corresponds to a left $E(n)$-module structure on $\AE{n}$. Since we will focus only on the case when $n=2$, we will only write this part down explicitly. The $E(2)$-action on $\AE{2}$ is given by 
\[
Q_i\zeta_n = 0
\]
for all $i$ and $n$, and for all $k\geq 1$,
\begin{align}\label{eqn:Q0action}
Q_0\otau_{2+k} &= \zeta_{2+k}\\
\label{eqn:Q1action}Q_1\otau_{2+k} &= \zeta_{1+k}^p\\
\label{eqn:Q2action}Q_2\otau_{2+k} &= \zeta_{k}^{p^2}
\end{align}

These can be obtained from Proposition 17.10 of \cite{Switzer} as follows. Let $X$ be a spectrum. Then the mod $p$ homology of $X$ is given by 
\[
H_*(X):= \pi_*(H\wedge X).
\]
This is naturally a left comodule over the dual Steenrod algebra. But it is also a left $A$-module: Let $a$ be an element of the mod $p$ Steenrod algebra, then the action of $a$ on $H_*X$ arises from 
\[
\begin{tikzcd}
	H\wedge X\arrow[r, "a\wedge 1_X"] & H\wedge X.
\end{tikzcd}
\]
This left action of $A$ and the left comodule structure on $H_*X$ are intricately related. 

\begin{prop}[17.10, \cite{Switzer}]
	Let $a\in A$ and $u\in H_*X$. Suppose that the coaction on $u$ is 
	\[
	\alpha(u) = \sum_i \eta_i\otimes u_i.
	\]
	Then there is the identity
	\[
	a \cdot u = \sum_i\langle a, \chi \eta_i\rangle u_i.
	\]
\end{prop}

\subsection{Margolis homology}

Now that we have calculated the (co)module structure on $\AE{2}$, we can compute its Margolis homology. This follows as in \cite{bluebook}. We let $M_*(\tBP{2};Q_i)$ denote the Margolis homology with respect to $Q_i$. Let $T_n(x)$ denote the $p^n$th truncated polynomial algebra generated by $x$, i.e. 
\[
T_n(x):= \F_p[x]/x^{p^n}.
\]

\begin{thm}\label{thm:Margolis homology}
	The Margolis homology for $\tBP{2}$ is given as 
	\begin{itemize}
		\item $M_*(\tBP{2};Q_ 0)\cong \F_p[\zeta_1, \zeta_2]$;
		\item $M_*(\tBP{2};Q_1)\cong \F_p[\zeta_1]\otimes T_1(\zeta_2, \zeta_3, \ldots)$; and 
		\item $M_*(\tBP{2};Q_2)\cong T_2(\zeta_1, \zeta_2, \ldots)$.
	\end{itemize}
\end{thm}
\begin{proof}
	This follows in the same fashion as \cite[Lemma 16.9, pg 341]{bluebook} or \cite[Proposition 2.13]{Culver2017}
\end{proof}

\section{Structural Results}\label{sec:structural results}

In \cite{Culver2017}, the author showed that there is a decomposition $\AE{2}$ as a $E(2)_*$-comodule, 
\[
\AE{2}\cong_{E(2)_*}S\oplus Q
\]
where $S$ has trivial Margolis homology, and is hence free. The author then argued that the Ext groups of $Q$ are $v_2$-torsion free by applying a Bockstein spectral sequence. In this section, we will perform the same steps. The purpose of this section is to carry out the corresponding arguments to the case of an odd prime $p$. This requires an appropriate modification of the the notion of \emph{length} from \cite{Culver2017}. Given this, most of the arguments carry over in a straightforward fashion.

\subsection{$E(2)_*$-comodule splitting}

In \cite{Culver2017}, the concept of \emph{length} was essential to producing the $E(2)_*$-comodule splitting of $H_*\tBP{2}$. We adapt these methods to the odd primes.

\begin{defn}
	We define the \emph{length} of a monomial $x$, denoted as $\ell(x)$, by setting 
	\[
	\ell(\zeta_n)=0
	\]
	and setting
	\[
	\ell(\otau_n)=1
	\]
	for all $n$, and extending $\ell$ multiplicatively. We define $F^\ell(\AE{2})$ to be the subspace spanned by monomials of length $\leq \ell$. We let $\AE{2}^{(\ell)}$ denote the subspace spanned by monomials of length $\ell$. If $x\in \AE{2}^{(k)}$, then we say $x$ has length $\ell$, and write $\ell(x)=k$.
\end{defn}

\begin{rmk}
	Suppose the monomial $m$ is given as $\zeta_1^{i_1}\zeta_2^{i_2}\zeta_3^{i_3}\otau_3^{\epsilon_3}\zeta_4^{i_4}\cdots$. Then the length is given by 
	\[
	\ell(m) = \#\{n\geq 3\mid \epsilon_n=1\}.
	\]
\end{rmk}

\begin{rmk}
	In \cite{Culver2017}, we defined the length of a monomial $x$ to be the number of $\zeta_i$'s in $x$ with an odd exponent. Our definition here may seem different at first. However, if one regards $\zeta_n$ as the even analogue of $\otau_n$ and and $\zeta_n^2$ as the even analogue of $\zeta_n$, then the definitions coincide.
\end{rmk}

\begin{obs}
	The $F^\ell(\AE{2})$ gives a filtration of $\AE{2}$ by subcomodules.
\end{obs}

As in \cite{Culver2017}, the Leibniz formula for the action by $Q_i$ shows that 

\begin{prop}\label{prop: Qs lower length by 1}
	For $i=0,1,2$, if $\ell(x)=k>0$, then $\ell(Q_ix)=k-1$.
\end{prop}

From this lemma, this allows us to put an extra grading on the Margolis homology of $\tBP{2}$. More precisely, length can be regarded as a grading on $\AE{2}$. In this grading, the action by $Q_i$ produces the following chain complex
\[
\begin{tikzcd}
	\cdots \arrow[r] & \AE{2}^{(\ell+1)}\arrow[r, "Q_i"] & \AE{2}^{(\ell)}\arrow[r] & \cdots \arrow[r, "Q_i"] & \AE{2}^{(0)}
\end{tikzcd}
\]
and the homology of this chain complex is $M_*(\AE{2};Q_i)$. This puts a bigrading on the Margolis homology 
\[
M_*(\tBP{2};Q_i) = \bigoplus_{\ell\geq 0}M_{*, \ell}(\tBP{2};Q_i).
\]
 Consequently, we have from \ref{thm:Margolis homology} that
\begin{cor}
	If $\ell>0$ then for $i=0,1,2$, we have 
	\[
	M_{*,\ell}(\tBP{2};Q_i)=0.
	\]
\end{cor}

Following \cite{Culver2017}, \textsection 2.2, we define 
\[
S:= E(2)\{x\in \AE{2}\mid \ell(x)\geq 3\}.
\]

The same proof as in the prime 2 case gives 

\begin{prop}
	The Margolis homology groups of $S$ are all trivial, consequently $S$ is a free $E(2)$-module. 
\end{prop}
\begin{proof}
	The proof is mutatis mutandis the proof of Proposition 2.21 in \cite{Culver2017}.
\end{proof}
This leads us to consider the following short exact sequence, 
\[
0\to S\to \AE{2}\to Q\to 0.
\]
Since the Margolis homology of $S$ is trivial, it follows that $S$ is free. By \cite[ch. 15, Theorem 27]{spectraMargolis}, a module over a finite Hopf algebra over a field $k$ (such as $E(2)$) is free if and only if it is injective. Consequently, we get a splitting of $\AE{2}$, 
\[
\AE{2}\cong_{E(2)}S\oplus Q.
\]
All that remains to show is that the Ext groups of $Q$ are $v_2$-torsion free. We do so in the next subsection.

\subsection{The $v_2$-Bockstein spectral sequence}

We proceed as in section 2 of \cite{Culver2017}. First, note that 
\[
E(2)\mmod E(1)_*\cong E(\otau_2)
\]
where, as an $E(2)_*$-comodule, the coaction on $\otau_2$ is 
\[
\alpha(\otau_2) = \otau_2\otimes 1+1\otimes \otau_2.
\]
Note that we have a short exact sequence of $E(2)_*$-comodules 
\begin{equation}\label{eq4.2:SES}
	0\to \F_p\to E(2)\mmod E(1)_*\to \Sigma^{2p^2-1}\F_p\to 0.
\end{equation}

\begin{prop}
	The connecting homomorphism for the short exact sequence \eqref{eq4.2:SES} induces multiplication by $v_2$ on $\Ext_{E(2)_*}$ (up to a sign).
\end{prop}

\begin{rmk}
	The only real difference in the proof below between the one found in \cite{Culver2017} is that we must keep track of signs.
\end{rmk}

\begin{proof}
	The short exact sequence \eqref{eq4.2:SES} gives a short exact sequence of cobar complexes, 
		\[
			0 \to C^\bullet_{E(2)_*}\F_p\to  C^\bullet_{E(2)_*}(E(2)\mmod E(1)_*)\to  C^\bullet_{E(2)_*}(\Sigma^{2p^2-1}\F_p)\to  0
		\]
		Let $z$ be a cocycle in the cobar complex for $\Sigma^{2p^2-1}\F_p$. So 
		\[
		z= \sum_i [a_{1i}\mid \cdots \mid a_{si}]
		\]
		For $z$ to be a cocycle, one must have 
		\[
		dz = \sum_i \sum_{j=1}^s(-1)^j[a_{1i}\mid \ldots \mid \psi(a_{ji})\mid \ldots \mid a_{si}]=0
		\]
		A lift of $z$ to the cobar complex for $E(2)\mmod E(1)_*$ is the element
		\[
		\widetilde{z} = \sum_i[a_{1i}\mid \ldots \mid a_{si}]\otau_2.
		\]
		The cobar differential on this lift is then 
		\begin{align*}
			d\widetilde{z} &= \sum_i \left( \sum_{j=1}^s(-1)^j[a_{1i}\mid \cdots \mid \psi(a_{ji})\mid \cdots \mid a_{si}]\otau_2 + (-1)^{s+1}[a_{1i}\mid \cdots \mid a_{si}\mid \otau_2]\right)\\
						  &= (-1)^{s+1}\sum_i [a_{1i}\mid \cdots \mid a_{si}\mid \otau_2]
		\end{align*}
		The second equality follows from the fact that $z$ is a cocycle. So the connecting homomorphism is the result of concatenations with $\otau_2$ in the cobar complex, which gives multiplication by $v_2$ (at least up to sign).
\end{proof}

This allows us to draw the following conclusion.

\begin{cor}
	In the derived category of $E(2)_*$-comodules, there is a distinguished triangle
	\[
	\begin{tikzcd}
		\Sigma^{2p^2-1}\F_p[-1]\arrow[r,"\cdot v_2"] & \F_p\arrow[r] & E(2)\sslash E(1)_*\arrow[r] & \Sigma^{2p^2-1}\F_p
	\end{tikzcd}
	\]
\end{cor}

Tensoring with $Q$ thus gives the following distinguished triangle
\[
\begin{tikzcd}
		\Sigma^{2p^2-1}Q[-1]\arrow[r,"\cdot v_2"] & Q\arrow[r] & Q\otimes E(2)\sslash E(1)_*\arrow[r] & \Sigma^{2p^2-1}Q
	\end{tikzcd}
\]
from which we can build the exact couple giving the $v_2$-BSS:
\[
E_1^{*,*,*}=\Ext^{*,*}_{E(1)}(Q)\otimes \F_p[v_2]\implies \Ext_{E(2)_*}(Q).
\]
with 
\[
E_1^{s,t,r} = \Ext_{E(1)}^{s,t}(Q)\otimes \F_p\{v_2^r\}.
\]
The group $E_\infty^{s,t,r}$ contributes to $\Ext^{s+r, t+(2p^2-1)r}_{E(2)_*}(Q)$ and the differentials are of the form 
\[
d_k:E_k^{s,t,r}\to E_k^{s-k+1,t-(2p^2-1)k, r+k},
\]
so that in Adams indexing, the differential looks like an Adams $d_1$-differential.

As in \cite{Culver2017} its enough to show that $\Ext_{E(1)}^{s,t}(Q)$ is concentrated in even $(t-s)$-degree. Note that there is a length filtration on $Q$ induced from the one on $\AE{2}$, and that all lengths are 0,1,2. Define $S'$ to be the $E(1)$-submodule of $Q$ generated by elements of length 2, i.e.
\[
S' := E(1)\{x\in Q\mid \ell(x)=2\}.
\]

Since the Margolis homology of $Q$ is isomorphic to the Margolis homology of $\AE{2}$, $Q$'s Margolis homology it is concentrated in  length 0 (and hence in even degree). We will argue the $Q_0$ and $Q_1$-Margolis homology of $S'$ is trivial. 

\begin{prop}
	The Margolis homology groups $M_*(S';Q_0)$ and $M_*(S';Q_1)$ are both trivial. Consequently, $S'$ is a free $E(1)$-module.
\end{prop}
\begin{proof}
	The proof is analogous to the one for Proposition 2.15 in \cite{Culver2017}. Suppose that $x\in S'$ is such that $Q_0x=0$, so $x$ is a $Q_0$-cycle. If $\ell(x)=0$, then the only way for $x\in S'$ is if there is a $y\in Q$ with $\ell(y)=2$ and $Q_0Q_1y=x$. So $x$ represents 0 in $M_*(S';Q_0)$ in this case. 
	
	Suppose then that $\ell(x)=1$. Since the Margolis homology of $Q$ is concentrated in length 0, it follows that there is a $y\in Q$ with $Q_0y=x$. Since $\ell(x)=1$, it follows that $\ell(y)=2$, and so $y\in S'$. Thus $x$ represents the zero class in $M_*(S';Q_0)$. Finally consider $\ell(x)=2$. It follows that $x=0$ in $Q$. So $M_*(S';Q_0)=0$. An analogous argument shows $M_*(S';Q_1)=0$.
\end{proof}

\begin{cor}
	Define $\overline{Q}$ by the short exact sequence of $E(1)_*$-comodules
	\[
	0\to S'\to Q\to \overline{Q}\to 0.
	\]
	Then, as $S'$ is free, and hence injective, we have a splitting of $E(1)_*$-comodules,
	\[
	Q\simeq_{E(1)_*} S'\oplus \overline{Q}.
	\]
\end{cor}

\begin{prop}\label{prop: S' even degrees}
	The $\Ext$-groups of $S'$ are concentrated in even degree.
\end{prop}
\begin{proof}
	As $S'$ is a free $E(1)$-module, its Ext groups are concentrated in $\Ext^0$, so we just need to check that $\Ext^0$ is concentrated in even $(t-s)$-degree. A set of generators of $S'$ are the images of $m\otau_i\otau_j$ with $i\neq j$ where $m$ is a monomial of length 0. Since the degree of $\otau_n$ is $2p^n-1$, it follows that $m\otau_i\otau_j$ is in even degree. Since the action of the $Q_i$ changes degree by an odd number, it follows that the bottom cells of the $E(1)$-module generated by (the image) of $m\otau_i\otau_j$ is in even degree. So $\Ext^0_{E(1)_*}(S')$ is concentrated in even degree.
\end{proof}

We are left with showing that the Ext-groups for $\overline{Q}$ are concentrated in even degree. In order to do this, it is necessary to use the Adams-Priddy calculation of the Picard group of stable $E(1)$-modules (\cite{uniquenessBSO}). A review of the necessary details can be found in section 2 of \cite{Culver2017}. In particular, we will show that $\overline{Q}$ splits as a direct sum of invertible $E(1)$-modules. First, let $R$ denote the subalgebra of $\AE{2}$ given by 
\[
R = P(\zeta_2, \zeta_3, \ldots)\otimes E(\otau_3, \otau_4, \ldots)\subseteq \AE{2}.
\]
Note that the monomials of $R$ are those whose weight\footnote{The concept of weight is reviewed below.} is divisible by $p^2$. Moreover, $R$ is a $E(1)_*$-subcomodule algebra of $\AE{2}$. 

\begin{lem}
	As an $E(1)_*$-comodule algebra, $\AE{2}$ decomposes as 
	\[
	\AE{2}\cong_{E(1)_*}P(\zeta_1)\otimes R.
	\]
\end{lem}

This along with Theorem \ref{thm:Margolis homology} allows us to compute the Margolis homology of $R$. 

\begin{lem}
	The $Q_0$- and $Q_1$-Margolis homology of $R$ is given by 
	\begin{enumerate}
		\item $M_*(R;Q_0)\cong P(\zeta_2)$,
		\item $M_*(R;Q_1)\cong T_1(\zeta_2, \zeta_3, \ldots)$.
	\end{enumerate}
\end{lem}
%
%

Observe that the weight filtration on $\AE{2}$ yields a filtration on $R$. In particular, we obtain a direct sum decomposition 
\[
R\cong_{E(1)_*}\bigoplus_k W_2(k)
\]
as an $E(1)_*$-comodule, where 
\[
W_2(k) := \F_p\{m\in R\mid \wt(m) = p^2k\}.
\]
The Margolis homology of $W_2(k)$ is the weight $p^2k$ piece of the Margolis homology of $R$. 

\begin{prop}
	The Margolis homology groups $M_*(W_2(k);Q_i)$ for $i=0,1$ are both one dimensional $\F_p$-vector spaces. Thus, $W_2(k)$ is an invertible $E(1)_*$-comodule.
\end{prop}
\begin{proof}
	As we have already mentioned, $M_*(W_2(k);Q_i)$ is the subspace of $M_*(R;Q_i)$ spanned by monomials of weight $p^2k$. When $i=0$, this implies that 
	\[
	M_*(W_2(k);Q_0) = \F_p\{\zeta_2^k\}.
	\]
	Now fix $k$, and let its base $p$-expansion be 
	\[
	k = a_0+pa_1+p^2a_2+\cdots.
	\]
	Then the monomial 
	\[
	m = \zeta_2^{a_0}\zeta_3^{a_1}\zeta_4^{a_2}\cdots 
	\]
	has weight $p^2k$ and is an element of $M_*(W_2(k);Q_1)$. Suppose that $x$ is a monomial 
	\[
	x = \zeta_2^{b_2}\zeta_3^{b_3}\zeta_4^{b_4}\cdots 
	\]
	of weight $p^2k$ in $M_*(R;Q_1)$. Then $b_j<p$ for all $j$. Moreover, we have the equality
	\[
	k = b_2+pb_3+p^2b_4+\cdots.
	\]
	The uniqueness of base $p$ expansions implies that $b_j = a_{j-2}$ for all $j\geq 2$. This shows that 
	\[
	M_*(W_2(k);Q_1) = \F_p\{\zeta_2^{a_0}\zeta_3^{a_1}\zeta_4^{a_2}\cdots\}.
	\]
\end{proof}

The decomposition of $\AE{2}$ gives a corresponding decomposition of $\overline{Q}$:
\[
\overline{Q}\cong P(\zeta_1)\otimes \overline{R}\cong P(\zeta_1)\otimes \left( \bigoplus_k \overline{W}_2(k)\right)
\]
where $\overline{R}$ is the quotient of $R$ by the $E(1)$-submodule generated by monomials of length at least 2, and $\overline{W}_2(k)$ are the corresponding quotients of $W_2(k)$. Note that the quotient map $R\to \overline{R}$ is an isomorphism on Margolis homology. Thus, the Margolis homology groups of $\overline{W}_2(k)$ are both one dimensional. So by \cite[Lemma 3.5]{uniquenessBSO}, these $E(1)_*$-comodules are invertible. We are now in a position to conclude the following. 

\begin{prop}\label{prop: Qbar even degrees}
	The $\Ext$-groups of $\overline{Q}$ are concentrated in even degree.
\end{prop}
\begin{proof}
	The proof of this statement is completely analogous to the $p=2$ statement found in \cite[Proposition 2.34]{Culver2017}. 
\end{proof}

The following theorem follows from Propositions \ref{prop: S' even degrees} and \ref{prop: Qbar even degrees}.

\begin{thm}
	The $v_2$-Bockstein spectral sequence for $Q$ collapses at $E_2$. Consequently, the module $\Ext_{E(2)}(Q)$ is $v_2$-torsion free and concentrated in even degree. 
\end{thm}

This gives the desired splitting on the $E_2$-term of the ASS, and gives the collapsing and topological splitting we had at the 2-primary case. 

\begin{cor}
The Adams spectral sequence $E_2$-term for $\tBP{2}\wedge \tBP{2}$ decomposes
\[
\Ext_{E(2)_*}(H_*\tBP{2})\cong_{\Ext_{E(2)_*}(\F_p)}\Ext_{E(2)_*}(S)\oplus \Ext_{E(2)_*}(Q).
\]	
In particular, it decomposes into a summand which is concentrated in $s=0$ and a summand which is $v_2$-torsion free and concentrated in even degrees. 
\end{cor}

Since $\tBP{2}\wedge \tBP{2}$ is a module spectrum over $\tBP{2}$, the differentials in its ASS are linear over the ASS for $\tBP{2}$. Hence we can conclude

\begin{cor}
The ASS for $\tBP{2}\wedge \tBP{2}$ collapses at $E_2$. 	
\end{cor}

In \cite[\textsection 2.6]{Culver2017}, the splitting of the Adams $E_2$-term (and hence of the homotopy groups) was seen to have a topological origin. The same proof gives an odd analogue. 

\begin{cor}
There is an equivalence of spectra
\[
\tBP{2}\wedge \tBP{2}\simeq C\vee HV
\]	
where $HV$ is a generalized Eilenberg-MacLane spectrum with 
\[
\pi_*(HV)\cong \Ext_{E(2)_*}(S)
\]
and
\[
\pi_*C\cong \Ext_{E(2)_*}(Q).
\]
\end{cor}

\section{Weight \& Brown-Gitler subcomodules}\label{sec:weight}

In \cite{Culver2017}, we defined a notion of weight for the comodules $\AE{n}$ and we defined the $i$th Brown-Gitler subcomodule of $\AE{n}$ to be the subspace spanned by monomials of weight $2i$. This was then used to produce exact sequences involving the Brown-Gitler subcomodules of $\AE{2}$ (cf. \cite[Lemmas 3.15 and 3.17]{Culver2017}). This was the entire basis of an inductive method for explicitly determining the Ext groups in the Adams $E_2$-term of $\tBP{2}\wedge \tBP{2}$. The purpose of this section is to modify these methods to the present situation. The reader who is familiar with \cite{Culver2017} will notice that several of the arguments are in fact simpler in the odd context. 

\subsection{Weight and Brown-Gitler subcomodules at odd primes}

Central to the computational methods of \cite{Culver2017} was the concept of weight. We begin with an odd prime analogue. 

\begin{defn}\label{defn:weight}
	Let 
	\[
	\wt(\otau_k):=p^k
	\]
	and 
	\[
	\wt(\zeta_k):= p^k
	\]
	and extend multiplicatively, 
	\[
	\wt(xy) = \wt(x)+\wt(y).
	\]
\end{defn}

\begin{rmk}
	As with length, if one makes the usual translation between the odd primary and mod 2 dual Steenrod algebra, this definition lines up with the notion of weight we used in \cite{Culver2017}.
\end{rmk}

\begin{defn}\label{defn:BrownGitlercomodule}
	The \emph{jth Brown-Gitler comodule}, denoted $N_i(j)$, is the subspace of $\AE{i}$ spanned by monomials of weight less than or equal to $pj$.
\end{defn}

\begin{rmk}
	For consistency with the notation of Brown-Gitler spectra, we will usually denote $N_i(j)$ with $\tBPu{i}_j$. In the particular case when $i=1$, then $\tBP{1}$ is equivalent to the Adams summand $\ell$, so we will write $\ellu_j$ for $N_1(j)$.
\end{rmk}

From the coproduct formulas \eqref{eqn2:coprod-zetas} and \eqref{eqn2:coprod-otaus}, we can observe that $\tBPu{i}_j$ is a $A_*$-subcomodule of $\AE{i}$.

\begin{defn}\label{defn:weightequalcomodule}
	Let $M_i(j)$ denote the subspace of $\AE{i}$ spanned by the monomials of weight exactly $pj$.
\end{defn}

In this section, we are most concerned with restricted coaction of $\AE{i}$ to $E(i)_*$. Recall that the $E(i)_*$-coaction on $\AE{i}$ is the following composite;
\[
\begin{tikzcd}
	\alpha_{E(i)}: \AE{i}\arrow[r, "\alpha"] & A_*\otimes \AE{i}\arrow[r,"\pi\otimes 1"] & E(i)_*\otimes \AE{i}
\end{tikzcd}
\]
We recorded the explicit formula for $\alpha_{E(i)}$ in Proposition \ref{prop:E(n)coaction}. This leads us to 

\begin{prop}
	Suppose that $x\in \AE{i}$ is a monomial and that 
	\[
	\alpha_{E(i)_*}(x) = 1\otimes x+ \sum_{j=1}^i \otau_j\otimes x_j
	\]
	for $x_j\in \AE{i}$. Then $\wt(x') = \wt(x)$. 
\end{prop}
\begin{proof}
	Since $\AE{i}$ is a comodule algebra, its enough to do this in the case that $x=\zeta_n$ or $x= \otau_n$. These cases follow from Proposition \ref{prop:E(n)coaction}.
\end{proof}

It follows from this proposition that $M_i(j)$ is a subcomodule of $\AE{i}$. Thus we have a direct sum decomposition
\[
\AE{i}\cong_{E(i)_*} \bigoplus_{k\geq 0}M_i(k)
\]
of $E(i)_*$-comodules.

When $i\geq 1$, consider the algebra map defined by 
\[
\varphi_i: \AE{i}\to \AE{i-1};\, \begin{cases}
	\zeta_k^{p^\ell}\mapsto \zeta_{k-1}^{p^\ell} & \forall k\\
	\otau_j\mapsto \otau_{j-1} & \forall j\geq i+1
\end{cases}
\]

\begin{prop}
	The map $\varphi_i$ is a map of ungraded $E(i)_*$-comodules. 
\end{prop}
\begin{proof}
	As in the $p=2$ case, observe that $\AE{i}$ is a $E(i)_*$-comodule algebra generated by the elements $\{\zeta_k\mid k\geq 1\}$ and $\{\otau_n\mid n\geq i+1\}$. The coproduct formulas \eqref{eqn:Enzeta} and \eqref{eqn:Enotau} imply the proposition. Keep in mind that we are regarding $\AE{i-1}$ as an $E(i)_*$-comodule.
\end{proof}

\begin{lem}
	The map $\varphi_i$ carries $M_i(j)$ isomorphically onto $N_{i-1}(\floor{j/p})$.
\end{lem}
\begin{proof}
	The proof is essentially the same as in the $2$-primary case. Let $x\in M_i(j)$ be a monomial, let's say it is 
	\[
	x = \zeta_1^{k_1}\zeta_2^{k_2}\cdots \zeta_{i+1}^{k_{i+1}}\otau_{i+1}^{\epsilon_{i+1}}\zeta_{i+2}^{k_{i+2}}\otau_{i+2}^{\epsilon_{i+2}}\cdots 
	\]
	Then 
	\[
	\wt(x) = \sum_{r\geq 1} k_rp^r + \sum_{s\geq i+1}\epsilon_s p^s
 	\]
 	and
 	\[
 	\varphi_i(x) = \zeta_1^{k_2}\cdots \zeta_i^{k_{i+1}}\otau_i^{\epsilon_{i+1}}\zeta_{i+1}^{k_{i+2}}\cdots.
 	\]
 	Thus
 	\[
 	\wt(\varphi_i(x)) = \sum_{r\geq 1}k_{r+1}p^r + \sum_{s\geq i}\epsilon_{s+1}p^s.
 	\]
 	In other words, we have 
 	\[
 	\wt(\varphi_i(x)) = \frac{pj-pk_1}{p} = j-k_1.
 	\]
 	Observe that $\wt(\varphi_i(x))$ is divisible by $p$. This shows that $\varphi_i$ maps the subspace of $M_i(j)$ spanned by monomials whose power of $\zeta_1$ is $k_1$ isomorphically onto $M_{i-1}(\frac{j-k_1}{p})$. Letting $k_1$ vary (over integers $k_1$ such that $j-k_1$ is congruent to 0 mod $p$) shows that $\varphi_i$ maps $M_i(j)$ isomorphically onto $N_{i-1}\left(\floor{j/p}\right)$.
\end{proof}

\begin{rmk}\label{rmk:inverse-map}
The inverse of $\varphi_i$ is given on monomials by 
\[
\varphi_i^{-1}: \zeta_1^{k_1}\zeta_2^{k_2}\cdots \zeta_{i+1}^{k_{i+1}}\otau_{i+1}^{\epsilon_{i+1}}\zeta_{i+2}^{k_{i+2}}\otau_{i+2}^{\epsilon_{i+2}}\cdots 
	\mapsto 
	\zeta_1^a\zeta_2^{k_1}\zeta_3^{k_2}\cdots \zeta_{i+2}^{k_{i+1}}\otau_{i+2}^{\epsilon_{i+1}}\zeta_{i+3}^{k_{i+2}}\otau_{i+3}^{\epsilon_{i+2}}\cdots 
\]	
where 
\[
a = j-p^{-1}\wt(\zeta_2^{k_1}\zeta_3^{k_2}\cdots \zeta_{i+2}^{k_{i+1}}\otau_{i+2}^{\epsilon_{i+1}}\zeta_{i+3}^{k_{i+2}}\otau_{i+3}^{\epsilon_{i+2}}\cdots)
\]
\end{rmk}

\begin{cor}\label{cor5:comod-iso}
There is a graded isomorphism of $E(i)_*$-comodules 
\[
M_i(j)\cong \Sigma^{qj}N_{i-1}(\floor{j/p})
\]	
\end{cor}

\begin{prop}\label{prop:upbylessthanp}
	There is an isomorphism of graded $E(i)_*$-comodules
	\[
	M_i(pj)\cong \Sigma^{qk} M_i(pj+k)
	\]
	induced by multiplication by $\zeta_1^k$.
\end{prop}
\begin{proof}
	Same as for \cite[Prop. 3.7]{Culver2017}.
\end{proof}

Because of this last corollary, we will always make the identification 
\[
M_2(pj)\cong \Sigma^{qpj}N_1(j).
\]

\begin{rmk}\label{rmk:decomposition}
	Since there is a decomposition of $E(2)_*$-comodules, we have 
	\[
	\AE{2}\cong_{E(2)_*}\bigoplus_{k=0}^\infty M_2(k)\cong_{E(2)_*}\bigoplus_{k=0}^\infty \Sigma^{qk}\ellu_{\floor{k/p}}.
	\]
\end{rmk}

\subsection{Exact sequences}

Recall that 
\[
\E2E1\cong E(\otau_2)
\]
and consider the linear map 
\[
\kappa: \AE{1}\to \AE{2}\otimes \E2E1
\]
which is defined on the monomial basis by 
\[
\zeta_1^{k_1}\zeta_2^{k_2}\otau_2^{\epsilon_2}\zeta_3^{k_3}\otau_3^{\epsilon_3}\zeta_4^{k_4}\cdots\mapsto  \zeta_1^{k_1}\zeta_2^{k_2}\zeta_3^{k_3}\otau_3^{\epsilon_3}\cdots \otimes \otau_2^{\epsilon_2}
\]
which takes the $\otau_2$ into $\E2E1$. We endow the right hand side with the diagonal coaction. This is not a map of $E(2)_*$-comodules, as can be seen in the following example, 

\begin{ex}
	In $\AE{1}$, the coaction on $\otau_2$ is 
	\[
	\alpha(\otau_2) = 1\otimes \otau_2+ \otau_0\otimes \zeta_2+\otau_1\otimes \zeta_1^3+\otau_2\otimes 1.
	\]
	On the other hand, the coaction of $\otau_2\otimes 1$ in $\AE{2}\otimes \E2E1$ is 
	\[
	\alpha(1\otimes \otau_2) = \otau_2\otimes 1\otimes 1+1\otimes 1\otimes \otau_2
	\]
\end{ex}

However, the map $\kappa$ is an isomorphism of $\F_p$-vector spaces. Analogously to the prime $p=2$ case, we define the following filtration, 
\[
F^j\AE{1}:= \kappa^{-1}\left( \bigoplus_{k\geq j}M_2(k)\otimes \E2E1\right).
\]
This is a decreasing filtration on $\AE{1}$ which induces a map on the associated graded
\[
E^0\kappa: E^0\AE{1}\to \AE{2}\otimes \E2E1
\]
as in \cite{Culver2017}.

\begin{rmk}
	Unlike in the $p=2$ case, the map $\kappa$ is actually a morphism of $\F_p$-algebras since $\otau_2$ is exterior in $\AE{1}$. So $E^0\kappa$ is also an isomorphism of algebras.
\end{rmk}

Note the following, 

\begin{obs}\label{obs5:bddweight}
	Let $x$ be a monomial in $\AE{1}$. If $x\in F^j\AE{1}$ then the weight of $x$ is bounded below by $pj$. If the exponent of $\otau_2$ in $x$ is 1, then the weight is bounded below by $pj+p^2$.
\end{obs}

Observe that the coproduct formulas in Proposition \ref{prop:E(n)coaction}  show that the filtration $F^\bullet\AE{1}$ is a filtration by $E(2)_*$-comodules.

\begin{lem}\label{lem: filtration is multiplicative}
	The filtration $F^\bullet\AE{1}$ is a multiplicative filtration, i.e. 
	\[
	F^i\AE{1}\cdot F^j\AE{1}\subseteq F^{i+j}\AE{1}.
	\]
	Thus $E^0\AE{1}$ is an $E(2)_*$-comodule algebra.
\end{lem}
\begin{proof}
	Let $x\in \AE{1}$ be a monomial. Then $x$ can be uniquely expressed as $m\otau_2^{\epsilon}$ where $m$ is a monomial in $\AE{2}$. Note that $x\in F^i\AE{1}$ if and only if $\wt(m)$ is at least $pi$. So let $x\in F^i\AE{1}$ and $x'\in F^j\AE{1}$ be monomials. Then 
	\[
	xx' = mm'\otau_2^{\epsilon+\epsilon'}.
	\]
	Observe that the weight of $mm'$ is at least $p(i+j)$. So $xx'\in F^{i+j}\AE{1}$.
\end{proof}

\begin{prop}
	The map $E^0\kappa$ is an isomorphism of $E(2)_*$-comodule algebras. 
\end{prop}
\begin{proof}

We know that $E^0\kappa$ is a bijection and a map of algebras. So it suffices to check that $E^0\kappa$ commutes with the coproduct on a generating set for $E^0\AE{1}$. One such generating set is 
\[
\{\zeta_1, \zeta_2, \otau_2, \zeta_3, \otau_3, \ldots\}.
\]
In $E^0\AE{1}$, the coaction on all of these generators are the same as in $\AE{1}$ except for $\otau_2$, which is a comodule primitive in $E^0\AE{1}$. Since $1\otimes \otau_2$ is a comodule primitive in $\AE{2}\otimes E(\otau_2)$, the coaction commutes with $E^0\kappa$. 
\end{proof}


Define quotients
\[
Q^j\AE{1} = \AE{1}/F^{j+1}\AE{1}
\]
Then there is a finite decreasing filtration on $Q^j\AE{1}$ which gives an isomorphism
\[
E^0\kappa:E^0Q^j\AE{1}\to  \tBPu{2}_j\otimes \E2E1.
\]

\begin{rmk}\label{rmk:ext-iso}
	Since there is a finite decreasing filtration on $Q^j\AE{1}$, we obtain a strongly convergent spectral sequence 
	\[
	E_2=\Ext_{E(2)_*}(\E2E1\otimes \tBPu{2}_j)\implies \Ext_{E(2)_*}(Q^{j}\AE{1}).
	\]
	Note that the spectral sequence is linear over $\F_p[v_0,v_1,v_2]$ and observe that the $E_2$-term is $\Ext_{E(1)_*}(\tBPu{2}_j)$. Since the $E_2$-term is a direct sum of $v_1$-torsion in $\Ext^0$ and a $v_1$-torsion free component concentrated in even degrees,   it follows that the spectral sequence collapses. Consequently, we often regard $Q^j\AE{1}$ as $\tBPu{2}_j\otimes \E2E1$.
\end{rmk}

We now develop the analogs of the exact sequences found in \cite{Culver2017}.

\begin{lem}\label{lem: exact sequences}
	Let $j\geq 1$ and let $i\in \{0,1,\ldots, p-1\}$. Then there is an exact sequence of $E(2)_*$-comodules 
	\begin{equation}\label{eqn: exact sequence}
	0\to \Sigma^{qpj}\ellu_j\otimes \ellu_{i}\to \ellu_{pj+i}\to Q^{pj-1}\AE{1}\to \bigoplus_{k=i+1}^{p-1}\Sigma^{\varphi(j,k)}\ellu_{j-1}\to 0.
	\end{equation}
	where 
	\[
	\varphi(j,k):= q(p(j-1)+k)+|\otau_2|
	\]
	When $i=p-1$, then we take the right-hand term to be 0, so that this exact sequence is actually a short exact sequence. 
	\end{lem} 
\begin{proof}

Consider the following composite 
\[
\begin{tikzcd}
	\psi:\ellu_{pj+i}\arrow[r,hook] & \AE{1}\arrow[r, two heads] & Q^{pj-1}\AE{1}
\end{tikzcd}
\]
where the first is the inclusion map and the second is the projection. Observe that this is a map of $E(2)_*$-comodules. We can describe this map on the associated graded. More explicitly, we have the diagram
\[
\begin{tikzcd}
	\ellu_{pj+i}\arrow[r,hook] \arrow[dr]& \AE{1}\arrow[r, two heads]\arrow[d, "\kappa"] & Q^{pj-1}\AE{1}\arrow[d, "\kappa"]\\
	& \AE{2}\otimes E(\otau_2)\arrow[r, two heads]&\tBPu{2}_{pj-1}\otimes E(\otau_2)
\end{tikzcd}
\]
and the bottom composite is given on monomials by 
\[
m\otau_2^\epsilon\mapsto 
\begin{cases}
		m\otimes \otau_2^{\epsilon} & \wt(m)\leq p^2j-p\\
		0 & \text{otherwise}
\end{cases}.
\]
Let $\widetilde{\psi}$ denote the bottom composite. Since $\kappa$ is an $\F_p$-isomorphism, to determine the kernel of $\psi$, it is enough to determine the kernel of $\widetilde{\psi}$. Observe that $m\otau_2^\epsilon$ is in the kernel if and only if 
\[
p^2j-p<\wt(m)\leq p^2j+pi.
\]
Note that, if $\epsilon=1$, then the weight of $m\otau_2$ is bounded below by $p^2j+p^2-p$, which is greater than $p^2+pi$. So $m\otau_2\notin \ellu_{pj+i}$, and so $\epsilon=0$. This shows that 
\[
\ker\psi = \bigoplus_{k=0}^iM_2(pj+k).
\]
Applying Corollary \ref{cor5:comod-iso} and Proposition \ref{prop:upbylessthanp} produces the following isomorphisms 
\begin{align*}
	\ker \psi &= \bigoplus_{k=0}^iM_2(pj+k)\\
			  &\cong \bigoplus_{k=0}^iM_2(pj)\cdot \{\zeta_1^k\}\\
			  &\cong \bigoplus_{k=0}^i \Sigma^{qpj}\ellu_j\cdot\{\zeta_1^k\}\\
			  &\cong \Sigma^{qpj}\ellu_j\otimes \ellu_i
\end{align*}
as desired. 

It remains to determine the cokernel of $\psi$. Note that when $i=p-1$, then the map $\widetilde{\psi}$ is surjective, and hence so is $\psi$. So in this case the cokernel is trivial and we get a short exact sequence. 

So assume that $i\neq p-1$. We will first identify $\coker\psi$ as an $\F_p$-vector space by determining $\coker \widetilde{\psi}$. Observe that the cokernel of $\widetilde{\psi}$ is given, as an $\F_p$-vector space, by 
\[
\mathrm{coker}\widetilde{\psi}=\bigoplus_{k=1}^{p-i-1}M_2(p(j-1)+i+k)\otimes \F_p\{\otau_2\}.
\]
To see this, note that if $\epsilon=0$, then $m\otimes 1$ is necessarily in the image of $\widetilde{\psi}$. Thus the cokernel will be spanned by elements of the form $m\otimes \otau_2$ for appropriate $m$ in $\AE{2}$. The element $m\otimes\otau_2$ will not be in the image of $\widetilde{\psi}$ if and only if $m\otau_2\notin \ellu_{pj+i}$. This is equivalent to the inequality
\[
p^2j+pi<\wt(m\otau_2) = \wt(m)+p^2
\]
Hence, $m\otimes \otau_2$ is not in the image of $\widetilde{\psi}$ if and only if the weight of the monomial $m$ satisfies the following inequality:
\[
p^2j-p^2+pi = p(p(j-1)+i)<\wt(m)\leq p^2j-p.
\]
The second inequality comes from the fact that $m\in \tBP{2}_{pj-1}$. This gives the declared $\F_p$-vector space.

We wish to show that there is an isomorphism of $E(2)_*$-comodules
\[
\mathrm{coker}\psi\cong_{E(2)_*}\bigoplus_{k=1}^{p-i-1}M_2(p(j-1)+i+k)\otimes \F_p\{\otau_2\}
\]
To show this, consider an element of $\coker\psi$ which is represented by the monomial $m\otau_2$ where $m$ is a monomial in $\AE{2}$. Since $\AE{1}$ is a comodule algebra, one has
\[
\alpha(m\otau_2) = \alpha(m)(\otau_2\otimes 1+\otau_0\otimes \zeta_2+\otau_1\otimes \zeta_1^p+1\otimes \otau_2).
\]
Let $t\otimes x$ denote a term occurring in $\alpha(m)$\footnote{the element $x$ need not be a monomial}. Note that 
\[
\wt(m) = \wt(x)
\]
and $x\in \AE{2}$ since $m\in \AE{2}$, and as
\[
\wt(m)\geq p^2(j-1)+pi+p
\]
it follows that $m,x\in F^{p(j-1)+i+1}\AE{1}$. Observe that $\zeta_2$ and $\zeta_1^p$ are both in $F^p\AE{1}$. Thus $x\zeta_2$ and $x\zeta_1^p$ are in $F^{pj+i+1}\AE{1}$. But as $Q^{pj-1}\AE{1} = \AE{1}/F^{pj}\AE{1}$, it follows that these are 0 in $Q^{pj-1}\AE{1}$, and so are zero in the cokernel. So the terms $\otau_0 t\otimes x\zeta_2$ and $\otau_1 t\otimes x\zeta_1^p$ are zero in $\alpha(m\otau_2)$. Consider the term $\otau_2t\otimes x$. Note that 
\[
\wt(m)\leq p^2j-p
\]
and so $m$ and $x$ are in the image of $\psi$, and so are 0 in the cokernel. Combining these observations shows that 
\[
\alpha(m\otau_2) = \alpha(m)(1\otimes \otau_2)
\]
establishing the desired comodule isomorphism. By another application of Corollary \ref{cor5:comod-iso}, we obtain the desired exact sequence.
	\end{proof}

\begin{rmk}\label{rmk:inductive map}
	We can explicitly describe the first map in the exact sequence above. From Remark \ref{rmk:inverse-map}, the first map 
	\[
	0\to \Sigma^{qpj}\ellu_j\otimes \ellu_i\to \ellu_{pj+i}
	\]
	is given by 
	\[
	\zeta_1^{i_1}\zeta_2^{i_2}\otau_2^{\epsilon_2}\zeta_3^{i_3}\cdots\otimes \{\zeta_1^k\mid 0\leq k\leq i\} \mapsto \zeta_1^a\zeta_2^{i_1}\zeta_3^{i_2}\otau_3^{\epsilon_2}\zeta_4^{i_3}\cdots\cdot\{\zeta_1^k\mid 0\leq k\leq i\}
	\]
	where
	\[
	a = pj-p^{-1}\wt(\zeta_2^{i_1}\zeta_3^{i_2}\otau_3^{\epsilon_2}\zeta_4^{i_3}\cdots).
	\]
\end{rmk}

\begin{rmk}
	Associated to these exact sequences are spectral sequences, which will be analyzed in the next section to give an inductive procedure for computing $\Ext_{E(2)_*}(\ellu_j)$. Since we are actually interested in $\Ext_{E(2)}(\AE{2})$, it follows from Remark \ref{rmk:decomposition} that we should compute $\Ext_{E(2)_*}(\Sigma^{qpj}\ellu_j)$.
\end{rmk}

\begin{rmk}
	The proof Lemma \ref{lem: exact sequences} can be adapted to the $p=2$ case. In particular, one can give a single argument to show the existence of the exact sequences developed in \cite{Culver2017}.
\end{rmk}

\section{Calculations and applications}\label{sec:calculations}

In this section, following, \cite{Culver2017} and \cite{BOSS}, we analyze the spectral sequences induced by the exact sequences produced in the previous section. We will apply this in particular to the prime $p=3$ and compute $\Ext_{E(2)_*}(\Sigma^{12j}\ellu_{j})$ for small values of $j$. We continue to let $q:= 2(p-1)p$.


\subsection{Inductive calculations}

We will now develop an inductive technique to calculate the $v_2$-torsion free component of $\Ext_{E(2)_*}(\ellu_j)$, based on the exact sequences developed in the previous section. The content of this section is adapted from \cite{Culver2017}. Following \cite{BOSS}, we regard the exact sequences from Lemma \ref{lem: exact sequences} as giving spectral sequences converging to $\Ext_{E(2)_*}(\ellu_{pj+i})$ for $i=0, \ldots, p-1$. As in \cite{BOSS}, we write 
\[
\bigoplus M_i[k_i]\implies M
\]
to denote the existence of a spectral sequence
\[
\bigoplus \Ext_{E(2)_*}^{s-k_i, t+k_i}(M_i)\implies \Ext^{s,t}_{E(2)_*}(M),
\]
and we abbreviate $M_i[0]$ by $M_i$. Thus, Lemma \ref{lem: exact sequences} shows that for each $i\in \{0, \ldots, p-1\}$, there is a spectral sequences of the form 
\begin{equation}\label{eqn:inductiveSS}
	\hspace{40pt}\left(\Sigma^{qpj}\ellu_j\otimes \ellu_i\right)\oplus Q^{pj-1}\oplus \left(\bigoplus_{k=i+1}^{p-1}\Sigma^{q(p(j-1)+k)+|\otau_2|}\ellu_{j-1}[1]\right)\implies \ellu_{pj+i}.
\end{equation}
Again, if $i=p-1$, then the term involving $\ellu_{j-1}[1]$ is actually 0.

In order to carry out this calculation, we need to be able to calculate the Ext groups of $Q^j\AE{1}$. As mentioned in Remark \ref{rmk:ext-iso}, there is an isomorphism 
\[
\Ext_{E(2)_*}(\E2E1\otimes \tBPu{2}_j)\cong \Ext_{E(1)_*}\tBPu{2}_j.
\]
So it is enough to compute $\Ext_{E(1)_*}(\tBPu{2}_j)$. This is what we will do first. This part is analogous to the corresponding calculation in \cite{Culver2017}.

\begin{lem}
	For any $j$, there are isomorphisms
	\[
	\tBPu{2}_j\cong_{E(2)_*}\bigoplus_{0\leq k\leq j}M_2(k)\cong_{E(2)_*}\bigoplus_{0\leq k\leq j}\Sigma^{qk}\ellu_{\floor{k/p}}\cong_{E(1)_*}\bigoplus_{k=0}^j\bigoplus_{m=0}^{\floor{k/p}}\Sigma^{2k+2m}\HZu_{\floor{m/p}}
	\]
\end{lem}
\begin{proof}
	The proof is the same as in \cite[Lemma 3.22]{Culver2017}.
\end{proof}

In order to utilize this decomposition as an $E(1)_*$-comodule, we need to know $\Ext_{E(1)_*}(\HZu_k)$ for each natural number $k$. Below, for a comodule $M$, $\Ext(M)^{\langle n\rangle}$ will denote the $n$th Adams cover of $\Ext(M)$, and $\alpha_p(k)$ will denote the sum of the digits in the $p$-adic expansion of $k$.

\begin{lem}[{\cite[Theorem 2 and Prop 3]{CulverNumerical}, \cite[Prop 16.3]{bluebook}}]
	Modulo torsion, one has the isomorphism
	\[
	\Ext_{E(1)_*}(M_1(k))/v_0\text{-tors} \cong \Ext_{E(1)_*}(\Sigma^{2(p-1)k}\F_p)^{\left\langle \frac{k-\alpha_p(k)}{p-1}\right\rangle}
	\]
	Consequently, from \ref{cor5:comod-iso}, there is an isomorphism
	\[
	\Ext_{E(1)_*}(\HZu_k)/v_0\text{-tors}\cong \Ext_{E(1)_*}(\F_p)^{\left\langle \frac{k-\alpha_p(k)}{p-1}\right\rangle}. 
	\] 
\end{lem}

From this lemma we have determined the $E(1)$-Ext groups of $\tBPu{2}_{pj-1}$. However, we also want to keep track of the names of the generators. This follows from the following proposition. This will be done by determining the $v_0$-inverted Ext groups. 

\begin{prop}[\cite{BOSS}, \cite{Culver2017}]
	We have the following isomorphism 
	\[
	v_0^{-1}\Ext_{E(1)_*}(\tBPu{2}_j)\cong \F_p[v_0^{\pm1}, v_1]\{\zeta_1^i\zeta_2^k\mid i+pk\leq j\}
	\]
\end{prop}
\begin{proof}
	The proof is an appropriate adaption of the proof of Proposition 3.26 of \cite{Culver2017}.
\end{proof}

\begin{rmk}
	As in \cite{Culver2017}, the purpose of the preceding proposition is to keep track of the generators in lowest degree of the Adams cover associated to the integral Brown-Gitler comodules $\HZu_k$. Moreover, as in \cite[Remark 3.28]{Culver2017}, there is an algorithm to recover $\Ext_{E(1)_*}(\HZu_k)/tors$ from $v_0^{-1}\Ext_{E(1)_*}(\HZu_k)$.
\end{rmk}

After inverting $v_0$ in the spectral sequences \eqref{eqn:inductiveSS}, the $E_1$-terms become concentrated in even degree. Consequently, there are no differentials in the $v_0$-inverted spectral sequences. Similar arguments found in \cite{Culver2017} show the following

\begin{prop}
	In the spectral sequences \eqref{eqn:inductiveSS} the only nontrivial differentials occur between $v_2$-torsion classes. Consequently, when $i=p-1$, the spectral sequence collapses at $E_1$. 
\end{prop}

Since the $v_0$-inverted spectral sequences collapse immediately, we obtain for $i=0,\ldots, p-1$ summands 
\begin{equation}\label{eqn:BP2summand}
	v_0^{-1}\Ext_{E(1)_*}(\Sigma^{q(p^2j+pi)}\tBPu{2}_{pj-1})\subseteq v_0^{-1}\Ext_{E(2)_*}(\Sigma^{q(p^2j+pi)}\ellu_{pj+i}).
\end{equation}
Our first task is to provide the names of these generators. Following \cite[\textsection 3.3]{Culver2017}, we use the following diagram
\[
\begin{tikzcd}
	\Sigma^{q(p^2j+pi)}\ellu_{pj+i}\arrow[r]\arrow[d, "\cong"] & \Sigma^{q(p^2j+pi)}Q^{pj-1}\AE{1}\\
	M_2(p^2j+pi) & 
\end{tikzcd}.
\]
Given a monomial $\zeta_1^{i_1}\zeta_2^{i_2}\in v_0^{-1}\Ext_{E(1)_*}(\tBPu{2}_{pj-1})$, we obtain from the diagram, 
\[
\begin{tikzcd}
	\zeta_1^{i_1}\zeta_2^{i_2}\arrow[r, mapsto]\arrow[d, mapsto] & \zeta_1^{i_1}\zeta_2^{i_2}\\
\zeta_1^a\zeta_2^{i_1}\zeta_3^{i_2}
\end{tikzcd}
\]
where 
\[
a:= p^2j+pi-pi_1-p^2i_2.
\]
Thus we obtain
\begin{lem}\label{lem: BP2summand}
	Let $i\in \{0,\ldots, p-1\}$. The summands \eqref{eqn:BP2summand} are generated as modules over $\F_p[v_0^{\pm 1},v_1]$ by the monomials
	\[
	\zeta_1^{a}\zeta_2^{i_2}\zeta_3^{i_3}
	\]
	where $i_2+pi_3\leq pj-1$ and $a:= p^2j+pi-pi_2-p^2i_3$.
\end{lem}

We also have the following summands of the $v_0^{-1}$-Ext groups arising from the inductive terms,
\begin{equation}\label{eqn: inductive summand}
	v_0^{-1}\Ext_{E(2)_*}(\Sigma^{q((p^2+p)j+pi)}\ellu_j\otimes \ellu_i)\subseteq v_0^{-1}\Ext_{E(2)_*}(\Sigma^{q(p^2j+pi)}\ellu_{pj+i})
\end{equation}
for $i\in \{0,\ldots , p-1\}$. Our next goal is to name the generators of this summand. This is determined by considering the complementary parts of the exact sequences in \ref{lem: exact sequences}. Specifically, consider the diagram
\[
\begin{tikzcd}
	0\arrow[r]& \Sigma^{q((p^2+p)j+i)} \ellu_{j}\otimes \ellu_i\arrow[r] & \Sigma^{q(p^2j+pi)}\ellu_{pj+i}\arrow[d,"\cong"]\\
	& \Sigma^{qpj}M_2(p^2j)\otimes \ellu_i\arrow[u, "\cong"]& M_2(p^2j+pi)
\end{tikzcd}
\]
which on a monomial gives
\[
\begin{tikzcd}
	\zeta_1^{i_2}\zeta_2^{i_3}\otau_2^{\epsilon_3}\cdots \otimes \zeta_1^k\arrow[r, mapsto]& \zeta_1^{a+k}\zeta_2^{i_2}\zeta_3^{i_3}\otau_3^{\epsilon_3}\cdots\arrow[d, mapsto]\\
	\zeta_1^{i_1}\zeta_3^{i_2}\zeta_4^{i_3}\otau_4^{\epsilon_3}\cdots \otimes \zeta_1^k \arrow[u, mapsto]& \zeta_1^b\zeta_2^{a+k}\zeta_3^{i_2}\otau_4^{\epsilon_3}\cdots
\end{tikzcd}
\]
where the integer $a$ is as in Remark \ref{rmk:inductive map},
\[
a:= p^2j-p^{-1}\wt(\zeta_2^{i_2}\zeta_3^{i_3}\otau_3^{\epsilon_3}\cdots) = i_1.
\]
From Remark \ref{rmk:inverse-map}, the integer $b$ is given by 
\[
b=p^2j+pi-p^{-1}\wt(\zeta_2^{a+i_1}\zeta_3^{i_2}\otau_4^{\epsilon_3}\cdots) = p^2j+pi-(p^2j+pk)= p(i-k)
\]
Observe that in the case when $i=0$, then $k=0$, and so $b=0$.

%

 Combining these observations shows that,

\begin{prop}
	Assume inductively that $\Ext_{E(2)_*}(\Sigma^{qpj}\ellu_{j})$ has generators of the form $\{\zeta_1^{i_1}\zeta_2^{i_2}\zeta_3^{i_3}\cdots\}$. Then for $i=0$, the summand \eqref{eqn: inductive summand} has generators of the form $\{\zeta_2^{i_1}\zeta_3^{i_2}\cdots\}$ and for $0<i<p$, the summand \eqref{eqn: inductive summand} has generators of the form 
	\[
	\{\zeta_2^{i_1}\zeta_3^{i_2}\cdots\}\cdot\{\zeta_1^{pi}, \zeta_1^{p(i-1)}\zeta_2, \ldots, \zeta_1^p\zeta_2^{i-1},\zeta_2^i\}.
	\]
	
\end{prop}

We would now like to determine some hidden $v_2$-extensions in these spectral sequences. Most of them will be given by the \emph{length spectral sequence}, as developed in \cite{Culver2017}. We review this now spectral sequence now. 

There are other hidden $v_2$-extensions that need to be resolved. More specifically, the summands \eqref{eqn:BP2summand} are only modules over $\F_p[v_0^{\pm},v_1]$ on the $E_1$-page. To resolve these hidden extensions, we will make use of the \emph{length spectral sequence} as developed in \cite{Culver2017}. The starting point of this spectral sequence is that we can filter $\AE{2}$ by length. 

\begin{defn}
	Define a filtration on $\AE{2}$ by setting
	\[
	G^\lambda\AE{2}:= \F_p\{m\mid \ell(m)\leq \lambda\}
	\]
\end{defn}

Equations \eqref{eqn:Q0action}, \eqref{eqn:Q1action}, \eqref{eqn:Q2action} show that this is a filtration by $E(2)_*$-comodules. Furthermore, the filtration quotients are necessarily trivial, in light of Proposition \ref{prop: Qs lower length by 1}. Notice that $G^\bullet$ is an increasing filtration, which gives rise to the following unrolled exact couple:
\[
\begin{tikzcd}
	\Ext_{E(2)_*}(G^n\AE{2})\arrow[r] & \Ext_{E(2)_*}(G^{n+1}\AE{2}) \arrow[d] \\
	  & \Ext(G^{n+1}\AE{2}/G^n\AE{2}) \arrow[lu, dotted]
\end{tikzcd}.
\]
This yields a convergent spectral sequence
\[
E_1^{s,t,\lambda} = \Ext^{s,t}_{E(2)_*}(G^{\lambda}\AE{2}/G^{\lambda-1}\AE{2})\implies \Ext_{E(2)_*}^{s,t}(\AE{2})
\]
which we deem the \emph{length spectral sequence}. Since the filtration quotients are trivial $E(2)_*$-comodules, one finds that 
\[
E_1^{*,*,\lambda}\cong \F_p[v_0,v_1,v_2]\{m\mid \ell(m)=\lambda\}.
\]
Thus, the $E_1$-page is given by 
\[
E_1^{*,*,*}=E_0\AE{2}\otimes \F_p[v_0,v_1,v_2].
\]
Since the $d_1$-differential is essentially the connecting homomorphism in Ext groups, an elementary cobar complex argument shows that on a monomial $m$, one has
\[
d_1(m) = v_0\cdot Q_0m+v_1\cdot Q_1m+v_2\cdot Q_2m.
\]
From this we deduce that $\Ext_{E(2)_*}(\AE{2})$ has relations of the form 
\[
v_0Q_0(m)+v_1Q_1(m)+v_2Q_2(m)=0
\]
for any monomial $m$. In particular, if $m$ is of length 0 (i.e. consists only of $\zeta$'s), then we obtain from $d_1(m\otau_2)$ the relation
\[
v_2\zeta_1^{p^2}m+ v_1\zeta_2^pm+v_0\zeta_3m=0
\]
In particular, these relations imply that for a monomial generator $m$ of $\Ext$, we should have the following type of relations
\begin{equation}\label{eqn: lengthSS relations}
	v_2m + v_1\zeta_2^p\zeta_1^{-p^2}m+ v_0\zeta_3\zeta_1^{-p^2}m=0.
\end{equation}

Of course, this relation does not make sense if $\zeta_1^{p^2}$ does not divide the monomial $m$, and we will deal with these separately. 

\begin{ex}
	Consider the case when $p=3$ and $m = \zeta_1^9\zeta_3$. This is an element in $\Ext^0$ of 
	\[
	M_2(18)\cong \Sigma^{72}\ellu_6.
	\]
	The above implies that we have the following relation in $\Ext(\Sigma^{72}\ellu_6)$, 
	\[
	v_2\zeta_1^9\zeta_3+v_1\zeta_2^3\zeta_3+v_0\zeta_3^2=0.
	\]
	The reader might notice, based on the explicit $p=3$ computations performed in the next section, that each monomial appearing in this relation is a generator of $\Ext(\Sigma^{72}\ellu_6)$.
\end{ex}

\begin{rmk}
	In \cite{Culver2017}, we showed that if $m$ is a monomial generator of $\Ext^0$ arising from $Q^{2j-1}\AE{1}$ and for which $\zeta_1^8$ divides $m$, then $\zeta_1^{-8}\zeta_2^4 m$ and $\zeta_1^{-8}\zeta_3 m$, are also generators. This required a great deal of tedious, though elementary, arguments. The author strongly believes that similar statements can be made in the odd $p$ case, but refrains from doing so, as the arguments would be \emph{even more} tedious in the present case. 
	
	For the record, however, we provide the precise mathematical statement we believe to be true: if $m$ is a monomial generator which is an element in the summand \eqref{eqn:BP2summand} and if $\zeta_1^{p^2}$ divides $m$, then the monomials $\zeta_1^{-p^2}\zeta_2^pm$ and $\zeta_1^{-p^2}\zeta_3m$ are generators of $\Ext(\Sigma^{pqj}\ellu_j)$. As evidence, one can look at the explicit computations done in the next section and check directly.
\end{rmk}

As was already mentioned, one can only apply the relation \eqref{eqn: lengthSS relations} when $\zeta_1^{p^2}$ divides $m$. There are, however, monomial generators in \eqref{eqn:BP2summand} which are not divisible by $\zeta_1^{p^2}$. 

\begin{ex}
	Again, let $p=3$ and consider the exact sequence
	\[
	0\to \Sigma^{48}\ellu_1\to \Sigma^{36}\ellu_3\to Q^{2}\AE{1}\to \Sigma^{21}\ellu_0\oplus \Sigma^{25}\ellu_0\to 0.
	\]
	Then according to Lemma \ref{lem: BP2summand}, the monomials which generate the summand 
	\[
	v_0^{-1}\Ext_{E(1)_*}(\Sigma^{36}\tBPu{2}_2)\subseteq v_0^{-1}\Ext_{E(2)_*}(\Sigma^{36}\ellu_3)
	\] 
	are $\zeta_1^9, \zeta_1^6\zeta_2$ and $\zeta_1^3\zeta_2^2$. In this case, the job of the last term $\Sigma^{21}\ellu_0\oplus \Sigma^{25}\ellu_0$ is to receive the $v_2$-multiplications of $\zeta_1^6\zeta_2$ and $\zeta_1^3\zeta_2^2$. Note that these are the only generators in this summand for which the relation \eqref{eqn: lengthSS relations} does not apply.
\end{ex}

Before proceeding to state and prove the analogue of \cite[Corollary 3.38]{Culver2017}, we first locate the monomial generators of the summand \eqref{eqn:BP2summand} which are not divisible by $\zeta_1^{p^2}$. 

\begin{prop}
	Let $i\in \{0,\ldots, p-2\}$, and consider the exact sequence \eqref{eqn: exact sequence} for $\ellu_{pj+i}$. 
\end{prop}

We now need to determine some $v_2$-extensions in the spectral sequence \eqref{eqn:inductiveSS} in the case when $i=0, \ldots, p-2$. The odd primary analogue of \cite[Lemma 3.36]{Culver2017} will be useful 

\begin{lem}
	For a given $0\leq i<p$, let $M:= \bigoplus_{k=i+1}^{p-1}M_2(p(j-1)+k)$. Then the composite
	\[
	\begin{tikzcd}
		M\otimes E(\otau_2)\arrow[r, "\kappa"] & \AE{2}\arrow[r, "\pi"] & Q^{pj-1}\AE{1}
	\end{tikzcd}
	\]
	is a morphism of $E(2)_*$-comodules. 
\end{lem}
\begin{proof}
The proof is \emph{mutatis mutandis} the proof of \cite[Lemma 3.36]{Culver2017}, in particular it follows from the proof of Lemma \ref{lem: exact sequences}.
\end{proof}

\begin{rmk}
	Of course, when $i=p-1$, then $M=0$, and the lemma is trivially true. 
\end{rmk}

\begin{cor}
	Let $j\geq 1$ and let $0\leq i<p$. Let $M$ be as above. Then there is a the following commutative diagram in the category of $E(2)_*$-comodules: in the case when $i>0$,
	\[
	\begin{tikzcd}
		0\arrow[r] & \Sigma^{qpj}\ellu_{j}\otimes \ellu_i\arrow[r] & \ellu_{pj+i}\arrow[r]& Q^{pj-1}\AE{1}\arrow[r] &\bigoplus_{k=1}^{p-i-1}\Sigma^{\varphi(i,j,k)}\ellu_{j-1}\arrow[r] & 0\\
		0 \arrow[r] & 0 \arrow[r] & M\arrow[r]\arrow[u, hook] & M\otimes E(\otau_2)\arrow[r] \arrow[u]& M\otimes \F_p\{\otau_2\}\arrow[r] \arrow[u]& 0
	\end{tikzcd}
	\]
	where we define $\varphi(i,j,k)$ to be 
	\[
	\varphi(i,j,k) = q(p(j-1)+i+k)+2p^2-1.
	\]
\end{cor}

Before stating the analogue of \cite[Lemma 3.38]{Culver2017}, we give observe the following. 

\begin{lem}
	For the monomial generator 
	\[
	\zeta_1^a\zeta_2^{i_2}\zeta_3^{i_3}\in v_0^{-1}\Ext_{E(1)_*}(\Sigma^{q(p^2j+pi)}\tBPu{2}_{pj-1})\subseteq v_0^{-1}\Ext_{E(2)_*}(\Sigma^{q(p^2j+pi)}\ellu_{pj+i})
	\]
	the exponent $a$ is divisible $p$ and satisfies the inequality
	\[
	a\geq pi+p.
	\]
\end{lem}
\begin{proof}
	This follows from the description of $a$ and the bound on $i_2+pi_3$ in Lemma \ref{lem: BP2summand}.
\end{proof}
	
This lemma tells us that the there are $p-i-1$ different classes of monomial generators in the summand \ref{eqn:BP2summand} whose $v_2$-multiple needs to be resolved. They are the monomials of the form
\[
\zeta_1^{pi+p}\zeta_2^{i_2}\zeta_3^{i_3}, \zeta_1^{pi+2p}\zeta_2^{i_2}\zeta_3^{i_3}, \ldots, \zeta_1^{p^2-p}\zeta_2^{i_2}\zeta_3^{i_3}.
\]

\begin{cor}
	Let $k\in \{i+1, \ldots, p-1\}$. Consider the summand 
	\begin{equation*}
		\begin{split}
			v_0^{-1}\Ext_{E(1)_*}(\Sigma^{(q(p^2j+pi))}M_2(pj+i-k))&\subseteq v_0^{-1}\Ext_{E(1)_*}(\Sigma^{q(p^2j+pi)}\tBPu{2}_{pj-1})\\
				&\subseteq v_0^{-1}\Ext_{E(1)_*}(\Sigma^{q(p^2j+pi)}\ellu_{pj+i})
		\end{split}
	\end{equation*}
	generated over $\F_p[v_0^{\pm1},v_1]$ by monomials of the form 
	\[
	\zeta_1^{pk}\zeta_2^{i_2}\zeta_3^{i_3}
	\]
	with $i_2+pi_3 = pj+i-k$. In the summand 
	\[
	v_0^{-1}\Ext_{E(2)_*}(\Sigma^{\varphi(i,j,k-i)+q(p^2j+pi)}\ellu_{j-1}[1])\subseteq v_0^{-1}\Ext_{E(2)_*}(\Sigma^{q(p^2j+pi)}\ellu_{pj+i})
	\]
	let $x_{k,i}$ denote the generator corresponding to $\zeta_1^{i}$. Then in the $E_\infty$-page of the spectral sequence \eqref{eqn:inductiveSS} one has the following relation
	\[
	v_2\zeta_1^{pk}\zeta_2^{i_2}\zeta_3^{i_3} = x_{k,i_3}.
	\]
\end{cor}
\begin{proof}
	The proof proceeds in essentially the same way as \cite[Corollary 3.38]{Culver2017}. Let us just mention that, since $i-k\leq -1$, we have from \ref{cor5:comod-iso},
	\[
	M_2(pj+i-k)\cong \Sigma^{q(pj+i-k)}\cong \ellu_{j-1}
	\]
	since the floor of $(pj+i-k)/p$ is $j-1$.
\end{proof}

\subsection{Explicit computations at the prime $p=3$}

The purpose of this subsection is to explicitly carry out the inductive method of calculation produced in the previous section in the particular case when the prime is $3$. Though we will mostly be working at $p=3$, we will occasionally make comments for general $p$. In this setting, $q=4$. 

To begin, Lemma \ref{lem: exact sequences} tells us that for $j\geq 1$, we have the following three exact sequences, where we write $Q^{3j-1}$ for $Q^{3j-1}\AE{1}$, 
\begin{equation}\label{eqn: 0mod3 exact sequence}
\begin{tikzcd}[column sep = small]
	0\arrow[r] & \Sigma^{12j}\ellu_j\arrow[r] & \ellu_{3j}\arrow[r] & Q^{3j-1}\arrow[r] & \Sigma^{12j+9}\ellu_{j-1}\oplus \Sigma^{12j+13}\ellu_{j-1}\arrow[r] & 0
\end{tikzcd}
\end{equation}
\begin{equation}\label{eqn: 1mod3 exact sequence}
	\begin{tikzcd}[column sep = small]
			0\arrow[r] & \Sigma^{12j}\ellu_j\otimes \ellu_1\arrow[r] & \ellu_{3j+1}\arrow[r] & Q^{3j-1}\arrow[r] & \Sigma^{12j+13}\ellu_{j-1}\arrow[r] & 0
	\end{tikzcd}
\end{equation}
and 
\begin{equation}\label{eqn: 2mod3 exact sequence}
	\begin{tikzcd}[column sep = small]
		0\arrow[r] & \Sigma^{12j}\otimes \ellu_2\arrow[r] & \ellu_{3j+2}\arrow[r] & Q^{3j-1}\arrow[r] & 0
	\end{tikzcd}.
\end{equation}
This produces spectral sequences of the form 
\begin{equation}\label{eqn: 0mod3 SS}
	\Sigma^{48j}\ellu_j\oplus \Sigma^{36j}Q^{3j-1}\oplus \Sigma^{48j+9}\ellu_{j-1}[1]\oplus \Sigma^{48j+13}\ellu_{j-1}[1]\implies \Sigma^{36j}\ellu_{3j}
\end{equation}
\begin{equation}\label{eqn: 1mod3 SS}
	\Sigma^{48j+12}\ellu_j\otimes \ellu_1\oplus \Sigma^{36j+12}Q^{3j-1}\oplus \Sigma^{48j+25}\ellu_{j-1}[1]\implies \Sigma^{36j+12}\ellu_{3j+1}
\end{equation}
and
\begin{equation}\label{eqn: 2mod3 SS}
	\Sigma^{48j+24}\ellu_{j}\otimes \ellu_2\oplus \Sigma^{36j+24}Q^{3j-1}\implies \Sigma^{36j+24}\ellu_{3j+2}.
\end{equation}
The extra suspensions arise because our inductive method is actually computing the $\Ext_{E(2)_*}$-groups of the comodules $M_2(9j+3i)$ and we have the identification 
\[
M_2(9j+3i)\cong \Sigma^{36j+12i}\ellu_{3j+i}.
\]

Recall that 
\[
\AE{1} = P(\zeta_n\mid n\geq 1)\otimes E(\otau_k\mid k\geq 2)
\]
and that the weight is given by 
\[
\wt(\zeta_n) = \wt(\otau_n) = 3^n.
\]
Since the $\zeta_n$'s are $E(2)_*$-comodule primitives, the first few $\ellu_i$ are trivial $E(2)_*$-comodules. In particular, 

\begin{align*}
	\ellu_0&=\F_3\{1\},\\
	\ellu_1&=\F_3\{1,\zeta_1\},\\
	\ellu_2&=\F_3\{1, \zeta_1, \zeta_1^2\}.
\end{align*}
The first interesting $\ellu_i$ is $\ellu_3$, whose underlying $\F_3$-vector space is 
\[
\ellu_3 = \ellu_2\oplus \F_3\{\zeta_1^3, \zeta_2, \otau_2\}, 
\]
and the only non-primitive basis element is $\otau_2$. Recall that the coaction on $\otau_2$ is given by 
\[
\alpha(\otau_2) = 1\otimes \otau_2+\otau_0\otimes \zeta_2+\otau_1\otimes \zeta_1^3+\otau_2\otimes 1.
\]

\begin{exs}
We provide a few more examples of the modules $\ellu_j$ so that the Reader may divine a pattern:
\begin{align*}
	\ellu_4 &= \ellu_3\oplus \F_3\{\zeta_1^4, \zeta_1\zeta_2, \zeta_1\otau_2\},\\
	\ellu_5 &= \ellu_4\oplus \F_3\{\zeta_1^5, \zeta_1^2\zeta_2, \zeta_1^2\otau_2\},\\
	\ellu_6&= \ellu_5\oplus \F_3\{\zeta_1^6, \zeta_1^3\zeta_2, \zeta_1^3\otau_2, \zeta_2^2, \zeta_2\otau_2\},\\
	\ellu_7 &= \ellu_6\oplus \F_3\{\zeta_1^7, \zeta_1^4\zeta_2, \zeta_1^4\otau_2, \zeta_1\zeta_2^2, \zeta_1\zeta_2\otau_2\},\\
	\ellu_8 &= \ellu_7\oplus\{\zeta_1^8, \zeta_1^5\zeta_2, \zeta_1^5\otau_2, \zeta_1^2\zeta_2^2, \zeta_1^2\zeta_2\otau_2\},\\
	\ellu_9 &= \ellu_8\oplus \F_3\{\zeta_1^9, \zeta_1^6\zeta_2, \zeta_1^6\otau_2, \zeta_1^3\zeta_2^2, \zeta_1^3\zeta_2\otau_2, \zeta_2^3, \zeta_2^2\otau_2, \otau_3\}.
\end{align*} 	
The reader may notice from this pattern that $\zeta_n$ and $\otau_n$ will make their first appearance in the comodules $\ellu_{p^{n-1}}$.
\end{exs}

We now provide a tabulation of the generators which appear in the inductive spetral sequences \ref{eqn:inductiveSS}. In the table below, summands of the form $Q^{3j-1}\AE{1}$ are implicitly understood to be modules over $\F_3[v_0^{\pm},v_1]$. Terms with hidden $v_2$-extensions are indicated in red. 

\begin{alignat*}{3}
	\ellu_0: & \quad &\F_3: &\quad &&1\\
	\Sigma^{12}\ellu_1: && \Sigma^{12} \ellu_1: &&& \zeta_1^3, \zeta_2 \\
	\Sigma^{24}\ellu_2: && \Sigma^{24}\ellu_2: &&& \zeta_1^6, \zeta_1^3\zeta_2, \zeta_2^2\\
	\Sigma^{36}\ellu_3: && \Sigma^{36}Q^{2}\AE{1}: &&& \zeta_1^{9},\textcolor{red}{\zeta_1^6\zeta_2, \zeta_1^3\zeta_2^2}\\
	&& \Sigma^{48}\ellu_1: &&& \{\zeta_2^3, \zeta_3\}\\
	&& \Sigma^{57}\ellu_0[1]\oplus \Sigma^{62}\ellu_0[1]: &&& \textcolor{red}{v_2\zeta_1^6\zeta_2, v_2\zeta_1^3\zeta_2^2}\\
	\Sigma^{48}\ellu_4: && \Sigma^{48}Q^{2}\AE{1}: &&& \zeta_1^{12}, \zeta^9\zeta_2, \textcolor{red}{\zeta_1^6\zeta_2^2}\\
	&& \Sigma^{60}\ellu_1\otimes \ellu_1: &&& \{\zeta_2^3,\zeta_3\}\cdot\{\zeta_1^3,\zeta_2\}\\
	&& \Sigma^{69}\ellu_0[1]: &&& \textcolor{red}{v_2\zeta_1^6\zeta_2^2}\\
	\Sigma^{60}\ellu_5: && \Sigma^{60}Q^2\AE{1}: &&& \zeta_1^{15}, \zeta_1^{12}\zeta_2, \zeta_1^{9}\zeta_2^2\\
	&& \Sigma^{72}\ellu_1\otimes \ellu_2 &&& \{\zeta_2^3, \zeta_3\}\cdot\{\zeta_1^6, \zeta_1^3\zeta_2, \zeta_2^2\}\\
	\Sigma^{72}\ellu_6: && \Sigma^{72}Q^5\AE{1}: &&& \zeta_1^{18}, \zeta_1^{15}\zeta_2, \zeta_1^{12}\zeta_2^2, \zeta_1^9\zeta_2^3, \zeta_1^9\zeta_3, \textcolor{red}{\zeta_1^6\zeta_2^4, \zeta_1^6\zeta_2\zeta_3, \zeta_1^3\zeta_2^5, \zeta_1^3\zeta_2^2\zeta_3}\\
	&& \Sigma^{96}\ellu_2: &&& \zeta_2^6, \zeta_2^3\zeta_3, \zeta_3^2\\
	&& \Sigma^{105} \ellu_1[1]\oplus \Sigma^{109}\ellu_1[1]: &&& \textcolor{red}{v_2\zeta_1^6\zeta_2^4, v_2\zeta_1^6\zeta_2\zeta_3, v_2\zeta_1^3\zeta_2^5, \zeta_1^3\zeta_2^2\zeta_3} \\
	\Sigma^{84}\ellu_7: && \Sigma^{84}Q^5\AE{1}:&&& \zeta_1^{21}, \zeta_1^{18}\zeta_2, \zeta_1^{15}\zeta_2^2, \zeta_1^{12}\zeta_2^3, \zeta_1^9\zeta_2^4, \zeta_1^{12}\zeta_3, \textcolor{red}{\zeta_1^{6}\zeta_2^5}, \zeta_1^9\zeta_2\zeta_3, \textcolor{red}{\zeta_1^6\zeta_2^2\zeta_3}\\
	&& \Sigma^{108} \ellu_2\otimes \ellu_1: &&& \{\zeta_2^6, \zeta_2^3\zeta_3, \zeta_3^2\}\cdot\{\zeta_1^3, \zeta_2\}\\
	&& \Sigma^{117}\ellu_1[1]: &&& \textcolor{red}{v_2\zeta_1^6\zeta_2^5, \zeta_1^6\zeta_2^2\zeta_3}\\
	\Sigma^{96}\ellu_8: && \Sigma^{96}Q^5\AE{1}: &&& \zeta_1^{24}, \zeta_1^{21}\zeta_2, \zeta_1^{18}\zeta_2^2, \zeta_1^{15}\zeta_2^3, \zeta_1^{12}\zeta_2^4, \zeta_1^{15}\zeta_3, \zeta_1^9\zeta_2^5, \zeta_1^{12}\zeta_2\zeta_3, \zeta_1^9\zeta_2^2\zeta_3\\
	&& \Sigma^{120}\ellu_2\otimes \ellu_2 &&& \{\zeta_2^6, \zeta_2^3\zeta_3, \zeta_3^2\}\cdot \{\zeta_1^6, \zeta_1^3\zeta_2, \zeta_2^2\}\\
	\Sigma^{108}\ellu_9: && \Sigma^{108}Q^8\AE{1}: &&& \zeta_1^{27}, \zeta_1^{24}\zeta_2, \zeta_1^{21}\zeta_2^2, \zeta_1^{18}\zeta_2^3, \zeta_1^{15}\zeta_2^4, \zeta_1^{18}\zeta_3, \zeta_1^{12}\zeta_2^5, \zeta_1^{15}\zeta_2\zeta_3, \zeta_1^9\zeta_2^6,\\
	&& &&& \zeta_1^{12}\zeta_2^2\zeta_3,\textcolor{red}{\zeta_1^6\zeta_2^7}, \zeta_1^9\zeta_2^3\zeta_3, \textcolor{red}{\zeta_1^3\zeta_2^8, \zeta_1^6\zeta_2^4\zeta_3}, \zeta_1^9\zeta_3^2, \\
	&& &&& \textcolor{red}{\zeta_1^3\zeta_2^5\zeta_3, \zeta_1^6\zeta_2\zeta_3^2, \zeta_1^3\zeta_2^2\zeta_3^2}\\
	&& \Sigma^{144}Q^2\AE{1}: &&& \zeta_2^9, \zeta_2^6\zeta_3, \zeta_2^3\zeta_3^2\\
	&& \Sigma^{156}\ellu_1: &&& \zeta_3^3, \zeta_4\\
	&& \Sigma^{165}\ellu_0[1]\oplus \Sigma^{169}\ellu_0[1]: &&& v_2\zeta_2^6\zeta_3, v_2\zeta_2^3\zeta_3^2\\
	&& \Sigma^{153}\ellu_2[1]\oplus \Sigma^{157}\ellu_2[1]: &&& \textcolor{red}{v_2\zeta_1^6\zeta_2^7, v_2\zeta_1^3\zeta_2^8, v_2\zeta_1^6\zeta_2^4\zeta_3, v_2\zeta_1^3\zeta_2^5\zeta_3, \zeta_1^6\zeta_2\zeta_3^2, \zeta_1^3\zeta_2^2\zeta_3^2}.
\end{alignat*}

\newpage

\bibliographystyle{plain}
\bibliography{cooperations3}

\end{document}